\documentclass[11pt,a4paper]{article}
\usepackage{a4wide}
\usepackage{amsmath}
\usepackage{amssymb}
\usepackage{algorithm,algorithmic}
\usepackage{bm}
\usepackage{graphicx}
\usepackage[framed,amsmath,amsthm,thmmarks]{ntheorem}
\usepackage{framed}
\usepackage{color}
\usepackage{slashbox}
\usepackage{mathtools}

\allowdisplaybreaks

  \makeatletter
    
    \@addtoreset{algorithm}{section}
  \makeatother

\definecolor{bl}{gray}{0.9}

\def\FR{\mathop{\rm FR}\nolimits}

\def\diag{\mathop{\rm diag}\nolimits}

\def\grad{\mathop{\rm grad}\nolimits}

\def\D{{\rm D}}

\def\id{\mathop{\rm id}\nolimits}

\def\DY{{\rm DY}}
\def\sign{{\rm sign}}

\providecommand{\abs}[1]{\lvert#1\rvert}
\providecommand{\norm}[1]{\lVert#1\rVert}

\newtheorem{Def}{Definition}[section]
\newtheorem{Thm}{Theorem}[section]
\newtheorem{Prop}{Proposition}[section]

\newtheorem{Proof}{Proof}[section]
\newtheorem{Prob}{Problem}[section]
\newtheorem{Cor}{Corollary}[section]
\newtheorem{Assump}{Assumption}[section]

\newtheorem{predfn}{Definition}[section]
\newtheorem{prerem}[predfn]{Remark}

\newframedtheorem{mydef}[Def]{Definition}
\newframedtheorem{mythm}[Thm]{Theorem}
\newframedtheorem{myprop}[Prop]{Proposition}
\newframedtheorem{mylemma}[Lemma]{Lemma}
\newframedtheorem{myproof}[Proof]{Proof}
\newframedtheorem{myprob}[Prob]{Problem}

\title{A Dai--Yuan-type Riemannian conjugate gradient method\\ with the weak Wolfe conditions}
\author{Hiroyuki Sato\thanks{{\tt hsato@rs.tus.ac.jp}}\\
Department of Information and Computer Technology\\ Tokyo University of Science, Tokyo 125-8585, Japan}
\date{\today}

\begin{document}
\maketitle
\begin{abstract}
This article describes a new Riemannian conjugate gradient method and presents a global convergence analysis.
The existing Fletcher--Reeves-type Riemannian conjugate gradient method is guaranteed to be globally convergent if it is implemented with the strong Wolfe conditions.
On the other hand, the Dai--Yuan-type Euclidean conjugate gradient method generates globally convergent sequences under the weak Wolfe conditions.
This article deals with a generalization of Dai--Yuan's Euclidean algorithm to a Riemannian algorithm that requires only the weak Wolfe conditions.
The global convergence property of the proposed method is proved by means of the scaled vector transport associated with the differentiated retraction.
The results of numerical experiments demonstrate the effectiveness of the proposed algorithm.
\end{abstract}


\pagenumbering{arabic}
\section{Introduction}\label{sec1}

The Euclidean non-linear conjugate gradient method \cite{nocedal2006numerical} for minimizing a non-linear objective function $f:\mathbb R^n\to \mathbb R$ without any constraints is a generalization of the linear conjugate gradient method proposed by Hestenes and Stiefel \cite{hestenes1952methods}.
The steepest descent method, which is the simplest iterative optimization technique, does not need the Hessian of the objective function, but generally suffers from slow convergence.
Newton's method has a property of locally quadratic convergence, but this does not extend to global convergence.
Additionally, we need to compute the Hessian of the objective function at each step in Newton's method.
On the other hand, the conjugate gradient method ensures global convergence and is much faster than the steepest descent method.
Furthermore, it does not need the Hessian of the objective function.
Therefore, the conjugate gradient method is one of the most important optimization methods, and has been intensively researched.

The non-linear conjugate gradient method in Euclidean space $\mathbb R^n$ is characterized by its computation of search directions.
The search direction $\eta_k$ at the current iterate $x_k\in\mathbb R^n$ is computed by
\begin{equation}
\eta_k=-\nabla f(x_k)+\beta_k\eta_{k-1},\qquad k\ge 0,\label{ECGsearchvec}
\end{equation}
where $\beta_0=0$ and $\beta_k$ is a parameter that determines the property of the conjugate gradient method.
There are various choices of $\beta_k$, and a good choice leads to better convergence.
As with other line-search-based optimization methods, once a search direction is computed, the next iterate $x_{k+1}$ is computed by
\begin{equation}
x_{k+1}=x_k+\alpha_k\eta_k,\label{ls_euc}
\end{equation}
where the step size $\alpha_k>0$ is computed such that $\alpha_k$ approximately satisfies
\begin{equation*}
f(x_k+\alpha_k\eta_k)\approx\min_{\alpha>0}\{f(x_k+\alpha\eta_k)\}.
\end{equation*}
A frequently used rule for computing the step size is the Wolfe rule.
Under the Wolfe rule, $\alpha_k$ at the $k$-th iterate is computed such that $\alpha_k$ satisfies the Wolfe conditions
\begin{align}
f(x_k+\alpha_k\eta_k)\le f(x_k)+c_1\alpha_k \nabla f(x_k)^T\eta_k,\label{armijo}\\
\nabla f(x_k+\alpha_k\eta_k)^T\eta_k\ge c_2\nabla f(x_k)^T\eta_k\label{curvature}
\end{align}
for predetermined constants $c_1$ and $c_2$ with $0<c_1<c_2<1$.
In practice, $c_1$ and $c_2$ are often taken so as to satisfy $0 < c_1 < c_2 < 1/2$ in the conjugate gradient method.
To avoid confusion with the strong Wolfe conditions (in which the inequality \eqref{curvature} is replaced by a stricter condition $\abs{\nabla f(x_k+\alpha_k\eta_k)^T\eta_k}\le c_2\abs{\nabla f(x_k)^T\eta_k}$), we refer to the Wolfe conditions \eqref{armijo} and \eqref{curvature} as the weak Wolfe conditions.

We are interested in how to choose a good $\beta_k$ in \eqref{ECGsearchvec}.
A well-known choice proposed by Fletcher and Reeves \cite{fletcher1964function} is
\begin{equation*}
\beta_{k}^{\FR}=\frac{\nabla f(x_{k})^T\nabla f(x_{k})}{\nabla f(x_{k-1})^T\nabla f(x_{k-1})}. \label{fletcher_euc}
\end{equation*}
If the step sizes are computed so as to satisfy the strong Wolfe conditions, the conjugate gradient method with $\beta^{\FR}_k$ has global convergence.
In \cite{dai1999nonlinear}, Dai and Yuan proposed the following refinement of $\beta^{\FR}_k$:
\begin{equation}
\beta_k^{\DY}=\frac{\nabla f(x_k)^T\nabla f(x_k)}{\eta_{k-1}^Ty_k},\qquad y_k=\nabla f(x_k)-\nabla f(x_{k-1}).\label{betaintro}
\end{equation}
An advantage of $\beta^{\DY}_k$ is that it ensures the conjugate gradient method is globally convergent when implemented with only the weak Wolfe conditions.
There is no longer a need to assume that each step size satisfies the strong Wolfe conditions.

Beyond unconstrained optimization methods in Euclidean space, the idea of Riemannian optimization, or optimization on Riemannian manifolds, has recently been developed \cite{AbsMahSep2008,edelman1998geometry}.
Unconstrained optimization methods, such as the steepest descent method and Newton's method, have been generalized to Riemannian manifolds.
The conjugate gradient method has been generalized to Riemannian manifolds to some extent, but remains in the developmental stage.
In \cite{AbsMahSep2008}, Absil, Mahony, and Sepulchre introduced the notion of a vector transport to implement a Riemannian conjugate gradient method.
By means of this vector transport, Ring and Wirth performed a global convergence analysis of the Fletcher--Reeves-type Riemannian conjugate gradient method under the assumption that the vector transport as the differentiated retraction does not increase the norm of the search direction vector \cite{ring2012optimization}.
In \cite{sato2013cg}, Sato and Iwai introduced the notion of a scaled vector transport. This allowed them to develop an improved method, with a global convergence property that could be proved without the assumption made in \cite{ring2012optimization}.

The purpose of this article is to propose a new choice of $\beta_k$ for the Riemannian conjugate gradient method based on Dai--Yuan's $\beta^{\DY}_k$ in the Euclidean conjugate gradient method. We will also prove the global convergence property of the proposed algorithm under the weak Wolfe conditions.
Furthermore, we perform some numerical experiments to demonstrate that the proposed Dai--Yuan-type Riemannian conjugate gradient method is preferable to the existing Fletcher--Reeves-type method developed in \cite{sato2013cg},
and show that a step size satisfying the weak Wolfe conditions is easier to find than one that satisfies the strong Wolfe conditions.

This article is organized as follows.
In Section \ref{sec2}, we introduce several geometric objects necessary for Riemannian optimization.
We also define a Riemannian version of the weak Wolfe conditions, which are important in our algorithm.
In Section \ref{sec3}, we review the Dai--Yuan-type Euclidean conjugate gradient method, and discuss how to generalize $\beta^{\DY}_k$ to $\beta_k$ on a Riemannian manifold.
We take an approach based on another expression of \eqref{betaintro}, and propose a new algorithm.
Section \ref{sec4} provides a global convergence analysis of the present algorithm. This is analogous to a discussion in \cite{dai1999nonlinear}.
The notion of a scaled vector transport introduced in \cite{sato2013cg} plays an important role in our analysis.
In Section \ref{sec5}, we describe some numerical experiments that are intended to evaluate the performance of the proposed algorithm.
The results show that the proposed Dai--Yuan-type algorithm is preferable to the Fletcher--Reeves-type algorithm.
Our concluding remarks are presented in Section \ref{sec6}.

\section{General Riemannian optimization and Riemannian conjugate gradient method}\label{sec2}
In this section, we briefly review Riemannian optimization, especially the Riemannian conjugate gradient method.
Our problem is as follows.
\begin{Prob}\label{general_prob}
\begin{align*}
{\rm minimize} \,\,\,\,\,& f(x),\\
{\rm subject\,\,to} \,\,\,\,\,& x\in M,
\end{align*}
\end{Prob} where $M$ is a Riemannian manifold endowed with a Riemannian metric $\langle\cdot,\cdot\rangle$ and the norm of a tangent vector $\xi\in T_xM$ is defined to be $\norm{\xi}_x=\sqrt{\langle\xi,\xi\rangle_x}$, and where $f$ is a smooth objective function.
Note that $\langle\cdot,\cdot\rangle_x$ denotes the inner product on $T_x M$.

In Riemannian optimization, we have to replace several quantities used in Euclidean optimization with appropriate quantities on the Riemannian manifold $(M, \langle\cdot,\cdot\rangle)$ in question. 
For example, the search direction $\eta_k$ at the current point $x_k\in M$ must be a tangent vector to $M$ at $x_k$.
In iterative optimization methods, we perform a line search on appropriate curves on $M$.
Such a curve should emanate from $x_k$ in the direction of $\eta_k$, and can be defined by means of a retraction.
A retraction is defined as follows \cite{AbsMahSep2008}.
\begin{Def}\label{retractiondefgen}
Let $M$ and $TM$ be a manifold and the tangent bundle of $M$, respectively.
Let $R:TM\to M$ be a smooth map and $R_x$ be the restriction of $R$ to $T_xM$.
$R$ is called a retraction on $M$ if it has the following properties.
\begin{enumerate}
\item $R_x(0_x)=x$, where $0_x$ denotes the zero element of $T_xM$.
\item With the canonical identification $T_{0_x}T_xM\simeq T_xM$, $R_x$ satisfies
\begin{equation*}
{\rm D}R_x(0_x)=\id_{T_xM},
\end{equation*}
where $\D R_x(0_x)$ denotes the derivative of $R_x$ at $0_x$, and $\id_{T_xM}$ is the identity map on $T_xM$.
\end{enumerate}
\end{Def}
Using a retraction, the updating formula for line-search-based Riemannian optimization methods can be written as
\begin{equation}
x_{k+1}=R_{x_k}(\alpha_k\eta_k),\label{xupdate}
\end{equation}
where the step size $\alpha_k$ is computed so as to satisfy a certain condition.
Note that \eqref{xupdate} replaces \eqref{ls_euc}.
Throughout this article, we consider the weak Wolfe conditions
\begin{align}
f\left(R_{x_k}(\alpha_k\eta_k)\right)\le f(x_k)+c_1\alpha_k\langle\grad f(x_k),\eta_k\rangle_{x_k},\label{wolfem1}\\
\langle\grad f\left(R_{x_k}(\alpha_k\eta_k)\right),\D R_{x_k}\left(\alpha_k\eta_k\right)[\eta_k]\rangle_{R_{x_k}(\alpha_k\eta_k)}\ge c_2\langle\grad f(x_k),\eta_k\rangle_{x_k},\label{wolfem2}
\end{align}
where $0<c_1<c_2<1$ \cite{ring2012optimization,sato2013cg}.
Note that, on a general Riemannian manifold $M$, $\grad f$ is no longer the Euclidean gradient.
In fact, $\grad f$ is a vector field on $M$, and depends on the Riemannian metric.

In generalizing the Euclidean conjugate gradient method to that on a manifold $M$, the right-hand side of \eqref{ECGsearchvec} cannot be computed, since $\grad f(x_k)\in T_{x_k}M$ and $\eta_{k-1}\in T_{x_{k-1}}M$; that is, the two terms on the right-hand side 
of \eqref{ECGsearchvec} belong to different tangent spaces.

In \cite{AbsMahSep2008}, the notion of a vector transport was introduced to transport a tangent vector to another tangent space.
\begin{Def}\label{vt}
A vector transport $\mathcal{T}$ on a manifold $M$ is a smooth map
\begin{equation*}
TM\oplus TM\to TM:(\eta,\xi)\mapsto \mathcal{T}_{\eta}(\xi)\in TM
\end{equation*}
satisfying the following properties for all $x\in M$, where $\oplus$ is the Whitney sum \cite{AbsMahSep2008},
that is, $TM\oplus TM=\left\{(\eta,\xi)\,|\,\eta,\xi\in T_x M, x\in M\right\}$.
\begin{enumerate}
\item There exists a retraction $R$, called the retraction associated with $\mathcal{T}$, such that
\begin{equation*}
\pi\left(\mathcal{T}_{\eta}(\xi)\right)=R_x\left(\eta\right),\qquad \eta,\xi\in T_x M,
\end{equation*}
where $\pi\left(\mathcal{T}_{\eta}(\xi)\right)$ denotes the foot of the tangent vector $\mathcal{T}_{\eta}(\xi)$,
\item $\mathcal{T}_{0_x}(\xi)=\xi$ for all $\xi\in T_x M$,
\item $\mathcal{T}_{\eta}(a\xi+b\zeta)=a\mathcal{T}_{\eta}(\xi)+b\mathcal{T}_{\eta}(\zeta)$ for all $a, b\in \mathbb R,\ \eta,\xi,\zeta\in T_xM$.
\end{enumerate}
\end{Def}
We can generalize \eqref{ECGsearchvec} using a vector transport $\mathcal{T}$ on $M$ as
\begin{equation}
\eta_k=-\grad f(x_k)+\beta_k\mathcal{T}_{\alpha_{k-1}\eta_{k-1}}(\eta_{k-1}),\qquad k\ge 0,\label{MCGsearchvec}
\end{equation}
or equivalently,
\begin{equation}
\eta_k=-\grad f(x_k)+\beta_k\mathcal{T}^{(k-1)}_{\alpha_{k-1}\eta_{k-1}}(\eta_{k-1}),\qquad k\ge 0,\label{MCGsearchveck}
\end{equation}
where $\mathcal{T}^{(k)}:=c^{(k)}\mathcal{T}$ and $c^{(k)}$ is a positive number.
Thus, we have replaced $\beta_k$ with $\beta_kc^{(k)}$ in \eqref{MCGsearchvec} to obtain \eqref{MCGsearchveck}.
However, the latter expression \eqref{MCGsearchveck} is more useful in our discussion.
We will propose a new choice of $\beta_k$ in Section \ref{sec3}.

A reasonable choice of a vector transport is the differentiated retraction $\mathcal{T}^R$ defined by
\begin{equation}
\mathcal{T}^R_{\eta}(\xi):=\D R_x(\eta)[\xi],\qquad x\in M,\ \eta,\xi\in T_xM.\label{drpre}
\end{equation}
Note that the second condition \eqref{wolfem2} of the weak Wolfe conditions can be rewritten using $\mathcal{T}^R$ as
\begin{equation}
\langle\grad f\left(R_{x_k}(\alpha_k\eta_k)\right),\mathcal{T}^R_{\alpha_k\eta_k}(\eta_k)\rangle_{R_{x_k}(\alpha_k\eta_k)}\ge c_2\langle\grad f(x_k),\eta_k\rangle_{x_k}.\label{wolfem2eq}
\end{equation}
Furthermore, the scaled vector transport $\mathcal{T}^{0}$ associated with $\mathcal{T}^R$ \cite{sato2013cg}, which is defined by
\begin{equation}
\mathcal{T}^0_{\eta}(\xi)=\frac{\norm{\xi}_x}{\norm{\mathcal{T}^R_{\eta}(\xi)}_{R_x(\eta)}}\mathcal{T}^R_{\eta}(\xi),\qquad x\in M,\  \eta,\xi\in T_xM,\label{scaled}
\end{equation}
is important for analyzing the global convergence of our new algorithm.
Note that $\mathcal{T}^0$ is not a vector transport, as it does not satisfy the third condition of Def. \ref{vt}.
However, $\mathcal{T}^0$ has the important property that
\begin{equation*}
\norm{\mathcal{T}^0_{\eta}(\xi)}_{R_x(\eta)}=\norm{\xi}_x,\qquad \eta,\xi\in T_x M.
\end{equation*}

\section{Dai--Yuan-type Euclidean conjugate gradient method and its Riemannian generalization}\label{sec3}
\subsection{Dai--Yuan-type Euclidean conjugate gradient method}
The conjugate gradient method on $\mathbb R^n$ with $\beta_k^{\DY}$ defined by \eqref{betaintro}
was proposed by Dai and Yuan \cite{dai1999nonlinear}. This method is globally convergent under the assumption that each step size $\alpha_k$ satisfies the weak Wolfe conditions \eqref{armijo} and \eqref{curvature}.
Since the Fletcher--Reeves-type conjugate gradient method must be implemented with the strong Wolfe conditions, $\beta_k^{\DY}$ is an improved version of $\beta_k^{\FR}$.
We wish to develop a good analogy of $\beta_k^{\DY}$ for Riemannian manifolds.

Note that, in Euclidean space, we can show the equality
\begin{equation}
\beta_k^{\DY}=\frac{\nabla f(x_k)^T\eta_k}{\nabla f(x_{k-1})^T\eta_{k-1}}\label{EDY2}
\end{equation}
using Eq.~\eqref{ECGsearchvec} as
\begin{align*}
\beta_k^{\DY}=&\frac{\beta_k^{\DY}(\nabla f(x_k)-y_k)^T\eta_{k-1}}{\nabla f(x_{k-1})^T\eta_{k-1}}=\frac{\nabla f(x_k)^T(-\nabla f(x_k)+\beta_k^{\DY}\eta_{k-1})}{\nabla f(x_{k-1})^T\eta_{k-1}}\notag\\
=&\frac{\nabla f(x_k)^T\eta_{k}}{\nabla f(x_{k-1})^T\eta_{k-1}}.
\end{align*}
The equivalent expressions \eqref{betaintro} and \eqref{EDY2} for $\beta^{\DY}_k$ are useful in analyzing the global convergence of the Dai--Yuan-type algorithm in \cite{dai1999nonlinear}.

\subsection{New Riemannian conjugate gradient method based on Euclidean Dai--Yuan $\beta$}
Throughout this subsection, we assume that all quantities that appear in the denominator of a fraction are nonzero. 
However, this assumption can be removed after we propose a new algorithm at the end of this section;
see Prop.~\ref{lem1} for more details.
For simplicity, in some of the following computations, we use the notation
\begin{equation*}
g_k=\grad f(x_k),\qquad k\ge 0.
\end{equation*}

We expect a Riemannian analogy of the Dai--Yuan-type Euclidean conjugate gradient method to have global convergence if the step sizes satisfy the weak Wolfe conditions.

Let $\mathcal{T}$ be a general vector transport on $M$.
Assume that $\mathcal{T}^{(k)}:=c^{(k)}\mathcal{T}$, and use the update formula \eqref{MCGsearchveck} at the $k$-th iteration, where $c^{(k)}$ is a positive number.

Note that $\grad f(x_k)\in T_{x_k}M$ and $\grad f(x_{k-1})\in T_{x_{k-1}}M$ belong to different tangent spaces.
There are several possible ways of generalizing the right-hand side of \eqref{betaintro},
e.g., $\norm{g_k}_{x_k}^2/\langle\mathcal{T}^{(k-1)}_{\alpha_{k-1}\eta_{k-1}}(\eta_{k-1}),g_k-\mathcal{T}^{(k-1)}_{\alpha_{k-1}\eta_{k-1}}(g_{k-1})\rangle_{x_k}$ or $\norm{g_k}_{x_k}^2/\langle\eta_{k-1},(\mathcal{T}^{(k-1)}_{\alpha_{k-1}\eta_{k-1}})^{-1}(g_k)-g_{k-1}\rangle_{x_{k-1}}$.

On the other hand, it seems natural to generalize the right-hand side of \eqref{EDY2}, which is equivalent to \eqref{betaintro} in the Euclidean case, to
\begin{equation}
\beta_k:=\displaystyle\frac{\langle\grad f(x_k),\eta_k\rangle_{x_k}}{\langle\grad f(x_{k-1}),\eta_{k-1}\rangle_{x_{k-1}}}.\label{mybeta0}
\end{equation}
Note also that Eq.~\eqref{mybeta0}, which is an analogy of \eqref{EDY2} in the Euclidean case, is helpful in our global convergence analysis, as we will discuss later.
Therefore, we start with Eq.~\eqref{mybeta0} to generalize the Dai--Yuan $\beta_k$.
We should state that Eq.~\eqref{mybeta0} itself cannot be used in a conjugate gradient algorithm, since $\eta_k$ in the right-hand side of \eqref{mybeta0} is computed using $\beta_k$ itself, as in \eqref{MCGsearchveck}.
We wish to derive an expression for $\beta_k$ that does not contain $\eta_k$.
To this end, we obtain from \eqref{MCGsearchveck} and \eqref{mybeta0} that
\begin{align*}
\beta_k=&\displaystyle\frac{\langle g_k,-g_k+\beta_k\mathcal{T}^{(k-1)}_{\alpha_{k-1}\eta_{k-1}}(\eta_{k-1})\rangle_{x_k}}{\langle g_{k-1},\eta_{k-1}\rangle_{x_{k-1}}}\notag\\
=&\frac{-\norm{g_k}_{x_k}^2+\beta_k\langle g_k, \mathcal{T}^{(k-1)}_{\alpha_{k-1}\eta_{k-1}}(\eta_{k-1})\rangle_{x_k}}{\langle g_{k-1},\eta_{k-1}\rangle_{x_{k-1}}}.
\end{align*}
It follows that
\begin{equation}
\beta_k=\frac{\norm{\grad f(x_k)}_{x_k}^2}{\langle\grad f(x_k), \mathcal{T}^{(k-1)}_{\alpha_{k-1}\eta_{k-1}}(\eta_{k-1})\rangle_{x_k}-\langle\grad f(x_{k-1}),\eta_{k-1}\rangle_{x_{k-1}}}.\label{mybeta}
\end{equation}

We can further show that this $\beta_k$ in fact satisfies
\begin{equation}
\beta_k=\frac{\norm{\grad f(x_k)}_{x_k}^2}{\langle\mathcal{T}^{(k-1)}_{\alpha_{k-1}\eta_{k-1}}(\eta_{k-1}), y_k\rangle_{x_k}},\label{beta_revised}
\end{equation}
with $ y_k\in T_{x_k}M$ defined by
\begin{align}
y_k&=\grad f(x_k)\notag\\
&\ \ -\frac{\langle \grad f(x_{k-1}),\eta_{k-1}\rangle_{x_{k-1}}}{\langle\mathcal{T}^{(k-1)}_{\alpha_{k-1}\eta_{k-1}}(\grad f(x_{k-1})),\mathcal{T}^{(k-1)}_{\alpha_{k-1}\eta_{k-1}}(\eta_{k-1})\rangle_{x_k}}\mathcal{T}^{(k-1)}_{\alpha_{k-1}\eta_{k-1}}(\grad f(x_{k-1})).\label{deftildey}
\end{align}
This is because, from \eqref{deftildey}, it follows that
\begin{align*}
&\langle\mathcal{T}^{(k)}_{\alpha_{k}\eta_{k}}(\eta_{k}), y_{k+1}\rangle_{x_{k+1}}\notag\\
=&\left\langle\mathcal{T}^{(k)}_{\alpha_{k}\eta_{k}}(\eta_{k}), g_{k+1}-\frac{\langle g_{k},\eta_{k}\rangle_{x_{k}}}{\langle\mathcal{T}^{(k)}_{\alpha_{k}\eta_{k}}(g_{k}),\mathcal{T}^{(k)}_{\alpha_k\eta_k}(\eta_k)\rangle_{x_{k+1}}}\mathcal{T}^{(k)}_{\alpha_{k}\eta_{k}}(g_{k})\right\rangle_{x_{k+1}}\notag\\
=&\langle \mathcal{T}^{(k)}_{\alpha_k\eta_k}(\eta_k),g_{k+1}\rangle_{x_{k+1}}-\langle g_k,\eta_k\rangle_{x_k},\label{eq1}
\end{align*}
which implies that the denominators in the right-hand sides of \eqref{mybeta} and \eqref{beta_revised} are the same.

Equations \eqref{beta_revised} and \eqref{deftildey} would appear to be a natural generalization of \eqref{betaintro}.
However, $y_k$, and hence the right-hand side of \eqref{beta_revised}, are not always guaranteed to be well defined, since we cannot ensure that
\begin{equation}
\langle\mathcal{T}^{(k)}_{\alpha_{k}\eta_{k}}(\grad f(x_k)),\mathcal{T}^{(k)}_{\alpha_k\eta_k}(\eta_k)\rangle_{x_{k+1}} \neq 0
\label{neqassump}
\end{equation}
for all $k\ge 0$.
On the other hand, as discussed in Section \ref{sec4}, the right-hand side of Eq.~\eqref{mybeta} is well defined when we use step sizes satisfying the weak Wolfe conditions \eqref{wolfem1} and \eqref{wolfem2}.
Defining $\beta_k$ as \eqref{mybeta} has an advantage over \eqref{beta_revised}, because \eqref{mybeta} is valid without assumption \eqref{neqassump}.

Therefore, our strategy is to define a new $\beta_k$ using \eqref{mybeta}, rather than \eqref{beta_revised}.
Furthermore, similar to the Fletcher--Reeves-type Riemannian conjugate gradient method proposed in \cite{sato2013cg}, we use the scaled vector transport $\mathcal{T}^0$ associated with the differentiated retraction $\mathcal{T}^R$ only when $\mathcal{T}^R$ increases the norm of the search vector.
We now propose a new algorithm as Algorithm \ref{CGMDY}.

\begin{algorithm}[H]
\caption{A scaled Dai--Yuan-type Riemannian conjugate gradient method for Problem \ref{general_prob} on a Riemannian manifold $M$}
\begin{algorithmic}[1]\label{CGMDY}
\STATE Choose an initial point $x_0\in M$.
\STATE Set $\eta_0=-\grad f(x_0)$.
\FOR{$k=0,1,2,\ldots$}
\STATE Compute the step size $\alpha_k>0$ satisfying the weak Wolfe conditions \eqref{wolfem1} and \eqref{wolfem2} with $0<c_1<c_2<1$.
Set
\begin{equation*}
x_{k+1}=R_{x_k}\left(\alpha_k\eta_k\right),
\end{equation*}
where $R$ is a retraction on $M$.
\STATE Set
\begin{equation}
\beta_{k+1}=\frac{\norm{\grad f(x_{k+1})}_{x_{k+1}}^2}{\langle\grad f(x_{k+1}), \mathcal{T}^{(k)}_{\alpha_{k}\eta_{k}}(\eta_{k})\rangle_{x_{k+1}}-\langle\grad f(x_{k}),\eta_{k}\rangle_{x_{k}}},\label{betainalg}
\end{equation}
\begin{equation}
\eta_{k+1}=-\grad f(x_{k+1})+\beta_{k+1}\mathcal{T}^{(k)}_{\alpha_k\eta_k}(\eta_k),\label{tkeq}
\end{equation}
where $\mathcal{T}^{(k)}$ is defined by
\begin{equation}
\mathcal{T}^{(k)}_{\alpha_k\eta_k}(\eta_k)=\begin{cases}
\mathcal{T}^R_{\alpha_k\eta_k}(\eta_k),\qquad \text{if}\ \ \norm{\mathcal{T}^R_{\alpha_k\eta_k}(\eta_k)}_{x_{k+1}}\le \norm{\eta_k}_{x_k},
\\
\mathcal{T}^{0}_{\alpha_k\eta_k}(\eta_k),\qquad \text{otherwise},
\end{cases}\label{tkdef}
\end{equation}
and where $\mathcal{T}^R$ and $\mathcal{T}^0$ are the differentiated retraction and the associated scaled vector transport defined by \eqref{drpre} and \eqref{scaled}, respectively.
\ENDFOR
\end{algorithmic}
\end{algorithm}

Note that \eqref{tkdef} is well defined because, when we choose $\mathcal{T}^0$ at the $k$-th iteration, it holds that $\norm{\mathcal{T}^R_{\alpha_k\eta_k}(\eta_k)}_{x_{k+1}}>\norm{\eta_k}_{x_k}\ge 0$ and the quantity $\norm{\eta_k}_{x_k}/\norm{\mathcal{T}^R_{\alpha_k\eta_k}(\eta_k)}_{x_{k+1}}$ is well defined.
Furthermore, \eqref{tkdef} can also be written as
\begin{equation}
\mathcal{T}^{(k)}_{\alpha_k\eta_k}(\eta_k)=c^{(k)}\mathcal{T}^R_{\alpha_k\eta_k}(\eta_k), \quad c^{(k)}=\min\left\{1, \frac{\norm{\eta_k}_{x_k}}{\norm{\mathcal{T}^R_{\alpha_k\eta_k}(\eta_k)}_{x_{k+1}}}\right\}.\label{tkmin}
\end{equation}
Thus, we obtain the important inequality \cite{sato2013cg}
\begin{equation}
\norm{\mathcal{T}^{(k)}_{\alpha_k\eta_k}(\eta_k)}_{x_{k+1}}\le \norm{\eta_k}_{x_k},\label{Tkineq}
\end{equation}
which is essential in the global convergence analysis of our new algorithm.

\section{Global convergence analysis of the proposed new algorithm}\label{sec4}
In this section, we prove the global convergence property of the proposed algorithm.
We first describe our assumption about the objective function $f$.
\begin{Assump}\label{assump1}
The objective function $f$ is bounded below and of $C^1$-class, and there exists a Lipschitzian constant $L>0$ such that 
\begin{equation*}
\abs{\D(f\circ R_x)(t\eta)[\eta]-\D(f\circ R_x)(0)[\eta]}\le Lt,\ \eta\in T_xM,\ \norm{\eta}_x=1,\ x\in M,\ t\ge 0.\label{lip}
\end{equation*}
\end{Assump}
Examples of $f$ and the conditions under which this assumption holds are given in \cite{sato2013cg}.

We review a Riemannian analogy of Zoutendijk's theorem.
See \cite{ring2012optimization,sato2013cg} for more details.
\begin{Thm}\label{zoutendijk}
Suppose that a sequence $\{x_k\}$ on a Riemannian manifold $M$ is generated by a general line-search-based optimization algorithm; that is, by Eq.~\eqref{xupdate} with a retraction $R$ on $M$.
Suppose also that each search direction $\eta_k$ satisfies $\langle \grad f(x_k),\eta_k\rangle_{x_k}<0$ and that $\alpha_k$ satisfies the weak Wolfe conditions \eqref{wolfem1} and \eqref{wolfem2}.
If Assumption \ref{assump1} is satisfied, then the following series converges:
\begin{equation*}
\sum_{k=0}^\infty \cos^2\theta_k\norm{\grad f(x_k)}_{x_k}^2<\infty,\label{zoucos}
\end{equation*}
where $\cos\theta_k$ is defined by
\begin{equation*}
\cos\theta_k=-\frac{\langle\grad f(x_k),\eta_k\rangle_{x_k}}{\norm{\grad f(x_k)}_{x_k}\norm{\eta_k}_{x_k}}.
\end{equation*}
\end{Thm}

To show that $\beta_k$ in Algorithm \ref{CGMDY}, and hence the algorithm itself, is well defined, we prove that the search direction $\eta_k$ is a descent direction; that is, $\langle g_k,\eta_k\rangle_{x_k}<0$.
Note that in \cite{dai1999nonlinear}, the definition \eqref{betaintro} of $\beta_k^{\DY}$, which is a Euclidean analogy of the expression in~\eqref{beta_revised}, is fully used to prove that the Euclidean Dai--Yuan-type conjugate gradient method generates descent search directions.
In general, however, we cannot use \eqref{beta_revised}, as it is not well defined.
Therefore, we can only use \eqref{mybeta}.
Note that if $\langle g_k,\eta_k\rangle_{x_k}\neq 0$ for all $k \ge 0$, \eqref{mybeta} is equivalent to \eqref{mybeta0}.
In the proof of the next proposition, we use a different approach from that in \cite{dai1999nonlinear}.
\begin{Prop}\label{lem1}
If $\grad f(x_k)\neq 0$ for all $k\ge 0$, then Algorithm \ref{CGMDY} is well defined and the following two inequalities hold:
\begin{align}
\langle \grad f(x_k),\eta_k\rangle_{x_k}<0,\label{descent}\\
\langle \grad f(x_k),\eta_k\rangle_{x_k}<\langle \grad f(x_{k+1}),\mathcal{T}^{(k)}_{\alpha_k\eta_k}(\eta_k)\rangle_{x_{k+1}}.\label{ineq1}
\end{align}
\end{Prop}
\begin{proof}
The proof is by induction.
For $k=0$, the first inequality \eqref{descent} follows directly from $\eta_0=-g_0$.
We shall prove \eqref{ineq1} for $k=0$.
If $\langle g_1, \mathcal{T}^{(0)}_{\alpha_0\eta_0}(\eta_0)\rangle_{x_1}\ge 0$, it immediately follows that
\begin{equation}
\langle g_1, \mathcal{T}^{(0)}_{\alpha_0\eta_0}(\eta_0)\rangle_{x_1}\ge 0> \langle g_0,\eta_0\rangle_{x_0}.\label{ineq100}
\end{equation}
If $\langle g_1, \mathcal{T}^{(0)}_{\alpha_0\eta_0}(\eta_0)\rangle_{x_1}<0$,
and hence $\langle g_1,\mathcal{T}^R_{\alpha_0\eta_0}(\eta_0)\rangle_{x_1}< 0$,
then the second condition \eqref{wolfem2} of the weak Wolfe conditions with $0<c_2<1$ and \eqref{tkmin} gives
\begin{align}
\langle g_1,\mathcal{T}^{(0)}_{\alpha_0\eta_0}(\eta_0)\rangle_{x_1}=&\min\left\{1, \frac{\norm{\eta_0}_{x_0}}{\norm{\mathcal{T}^R_{\alpha_0\eta_0}(\eta_0)}_{x_{1}}}\right\}\langle g_{1},\mathcal{T}^R_{\alpha_0\eta_0}(\eta_0)\rangle_{x_{1}}\notag\\
\ge& \langle g_{1},\mathcal{T}^R_{\alpha_0\eta_0}(\eta_0)\rangle_{x_{1}}\ge c_2\langle g_0,\eta_0\rangle_{x_0}>\langle g_0,\eta_0\rangle_{x_0},\label{ineq10}
\end{align}
where we have used the fact that \eqref{wolfem2} is equivalent to \eqref{wolfem2eq}.
Thus, \eqref{ineq1} has been proved for $k=0$, and $\beta_1$ is well defined by \eqref{betainalg}.

Next, suppose that $\beta_k$ is well defined by the right-hand side of \eqref{mybeta}, and that both inequalities \eqref{descent} and \eqref{ineq1} hold for some $k$.
Note that $\beta_{k+1}$ is well defined from the assumption in \eqref{ineq1} for $k$.
The left-hand side of \eqref{descent} for $k+1$ can then be calculated as
\begin{align*}
&\langle g_{k+1},\eta_{k+1}\rangle_{x_{k+1}}=\langle g_{k+1},-g_{k+1}+\beta_{k+1}\mathcal{T}^{(k)}_{\alpha_k\eta_k}(\eta_k)\rangle_{x_{k+1}}\notag\\
=&-\norm{g_{k+1}}_{x_{k+1}}^2+\frac{\norm{g_{k+1}}_{x_{k+1}}^2}{\langle g_{k+1},\mathcal{T}^{(k)}_{\alpha_k\eta_k}(\eta_k)\rangle_{x_{k+1}}-\langle g_k,\eta_k\rangle_{x_k}}\langle g_{k+1},\mathcal{T}^{(k)}_{\alpha_k\eta_k}(\eta_k)\rangle_{x_{k+1}}\notag\\
=&\frac{\norm{g_{k+1}}_{x_{k+1}}^2\langle g_k,\eta_k\rangle_{x_k}}{\langle g_{k+1},\mathcal{T}^{(k)}_{\alpha_k\eta_k}(\eta_k)\rangle_{x_{k+1}}-\langle g_k,\eta_k\rangle_{x_k}}<0,
\end{align*}
where we have used inequalities \eqref{descent} and \eqref{ineq1} for $k$.
This means that \eqref{descent} holds for $k+1$.
It then follows from a similar argument to \eqref{ineq100} and \eqref{ineq10} that
\begin{equation*}
\langle g_{k+1},\eta_{k+1}\rangle_{x_{k+1}}<\langle g_{k+2},\mathcal{T}^{(k+1)}_{\alpha_{k+1}\eta_{k+1}}(\eta_{k+1})\rangle_{x_{k+2}}.
\end{equation*}
This completes the proof.
\end{proof}

The following corollary immediately follows from \eqref{ineq1}.
\begin{Cor}
In Algorithm \ref{CGMDY}, if $\grad f(x_k)\neq 0$ for all $k\ge 0$, then we have $\beta_k>0$ for all $k\ge 1$.
\end{Cor}

We now proceed to our main theorem.
The proof is performed in a manner analogous to the Euclidean version in \cite{dai1999nonlinear} with the aid of a scaled vector transport, as in \cite{sato2013cg}.
\begin{Thm}\label{globalthm}
Let $\{x_k\}$ be a sequence generated by Algorithm \ref{CGMDY}.
If Assumption \ref{assump1} is satisfied, then we have
\begin{equation}
\liminf_{k\to\infty}\norm{\grad f(x_k)}_{x_k}=0.\label{eqconv}
\end{equation}
\end{Thm}
\begin{proof}
We first note that, from Prop.~\ref{lem1}, all the assumptions in Thm.~\ref{zoutendijk} are satisfied.
Therefore, we have
\begin{equation}
\sum_{k=0}^{\infty}\frac{\langle g_k,\eta_k\rangle_{x_k}^2}{\norm{\eta_k}_{x_k}^2} < \infty.\label{zoutenineq}
\end{equation}

If $g_{k_0}=0$ for some $k_0$, then $\beta_{k_0}=0$, followed by $\eta_{k_0}=0$ and $x_{k_0+1}=R_{x_{k_0}}(0)=x_{k_0}$.
Thus, it is sufficient to prove \eqref{eqconv} only when $g_k\neq 0$ for all $k\ge 0$.
In such a case, it follows from \eqref{tkeq} that
\begin{equation*}
\eta_{k+1}+g_{k+1}=\beta_{k+1}\mathcal{T}^{(k)}_{\alpha_k\eta_k}(\eta_k).
\end{equation*}
Taking the norms and squaring, we obtain
\begin{equation*}
\norm{\eta_{k+1}}_{x_{k+1}}^2=\beta_{k+1}^2\norm{\mathcal{T}^{(k)}_{\alpha_k\eta_k}(\eta_k)}_{x_{k+1}}^2-2\langle g_{k+1}, \eta_{k+1}\rangle_{x_{k+1}}-\norm{g_{k+1}}_{x_{k+1}}^2.
\end{equation*}
This equality, together with \eqref{descent}, yields
\begin{align*}
&\frac{\norm{\eta_{k+1}}_{x_{k+1}}^2}{\langle g_{k+1}, \eta_{k+1}\rangle_{x_{k+1}}^2}
=\frac{\norm{\mathcal{T}^{(k)}_{\alpha_k\eta_k}(\eta_k)}_{x_{k+1}}^2}{\langle g_k, \eta_k\rangle_{x_k}^2}-\frac{2}{\langle g_{k+1}, \eta_{k+1}\rangle_{x_{k+1}}}-\frac{\norm{g_{k+1}}_{x_{k+1}}^2}{\langle g_{k+1}, \eta_{k+1}\rangle_{x_{k+1}}^2}\notag\\
=&\frac{\norm{\mathcal{T}^{(k)}_{\alpha_k\eta_k}(\eta_k)}_{x_{k+1}}^2}{\langle g_k, \eta_k\rangle_{x_k}^2}-\left(\frac{1}{\norm{g_{k+1}}_{x_{k+1}}}+\frac{\norm{g_{k+1}}_{x_{k+1}}}{\langle g_{k+1}, \eta_{k+1}\rangle_{x_{k+1}}}\right)^2+\frac{1}{\norm{g_{k+1}}_{x_{k+1}}^2}\notag\\
\le &\frac{\norm{\eta_k}_{x_{k}}^2}{\langle g_k, \eta_k\rangle_{x_k}^2}+\frac{1}{\norm{g_{k+1}}_{x_{k+1}}^2},
\end{align*}
where we have used the inequality \eqref{Tkineq} and the fact that $\beta_{k+1}$ is equal to the right-hand side of \eqref{mybeta0} for $k+1$, since $\langle g_k,\eta_k\rangle_{x_k}\neq 0$ from \eqref{descent}.
We thus arrive at the relation
\begin{equation}
\frac{\norm{\eta_{k}}_{x_{k}}^2}{\langle g_{k}, \eta_{k}\rangle_{x_{k}}^2}\le\sum_{i=1}^{k}\frac{1}{\norm{g_{i}}_{x_{i}}^2}+\frac{\norm{\eta_0}_{x_{0}}^2}{\langle g_0, \eta_0\rangle_{x_0}^2}=\sum_{i=0}^{k}\frac{1}{\norm{g_{i}}_{x_{i}}^2}.\label{eqthm}
\end{equation}

We are now in a position to prove \eqref{eqconv} by contradiction.
Given that we are now assuming $\grad f(x_k)\neq 0$ for all $k\ge 0$, if \eqref{eqconv} does not hold, then there exists a constant $C>0$ such that
\begin{equation}
\norm{g_k}_{x_k}\ge C,\qquad k\ge 0.\label{ineq2}
\end{equation}
From \eqref{eqthm} and \eqref{ineq2}, we obtain
\begin{equation*}
\frac{\norm{\eta_{k}}_{x_{k}}^2}{\langle g_{k}, \eta_{k}\rangle_{x_{k}}^2}\le\frac{k+1}{C^2}.
\end{equation*}
This leads to
\begin{equation*}
\sum_{k=0}^{N}\frac{\langle g_k,\eta_k\rangle_{x_k}^2}{\norm{\eta_k}_{x_k}^2}\ge C^2\sum_{k=0}^N\frac{1}{k+1}\to\infty,\qquad N\to \infty,
\end{equation*}
which contradicts \eqref{zoutenineq}.
This proves the theorem.
\end{proof}

\section{Numerical Experiments}\label{sec5}
In this section, we describe some numerical experiments designed to show the effectiveness of the proposed algorithm.
The experiments are carried out by using MATLAB R2014b on a PC with Intel Core i7-4790 3.60 GHz CPU, 16 GB of RAM memory and Windows 8.1 Pro 64-bit operating system.
In the experiments, we determine that a sequence converges to an optimal solution if $\norm{\grad f(x_k)}<\varepsilon$ for a predetermined tolerance $\varepsilon >0$.
We refer to $\beta_k$ in Algorithm \ref{CGMDY} as $\beta^{\DY}_k$, and to the Fletcher--Reeves-type $\beta_k$, which is discussed in \cite{sato2013cg}, as $\beta^{\FR}_k$, that is,
\begin{align*}
\beta^{\DY}_{k+1}=\frac{\norm{g_{k+1}}_{x_{k+1}}^2}{\langle g_{k+1}, \mathcal{T}^{(k)}_{\alpha_{k}\eta_{k}}(\eta_{k})\rangle_{x_{k+1}}-\langle g_{k},\eta_{k}\rangle_{x_{k}}},\quad
\beta^{\FR}_{k+1}=\frac{\norm{g_{k+1}}_{x_{k+1}}^2}{\norm{g_k}_{x_k}^2}.
\end{align*}

\subsection{Line search algorithms}
Let $\phi_k$ denote a one-variable function defined by $\phi_k(\alpha)=f(R_{x_k}(\alpha\eta_k))$, where $\eta_k$ is a search direction at $x_k\in M$.
Then, the Wolfe conditions \eqref{wolfem1} and \eqref{wolfem2} can be rewritten as
\begin{align}
\phi_k(\alpha_k)\le \phi_k(0)+c_1\phi_k'(0),\label{wolfe1ls}\\
\phi_k'(\alpha_k)\ge c_2\phi_k'(0).\label{wolfe2ls}
\end{align}
In terms of finding some $\alpha_k$ that satisfies \eqref{wolfe1ls} and \eqref{wolfe2ls},
the problem setting is the same as that in Euclidean optimization problems.
Therefore, we can use existing algorithms such as those in \cite{lemarechal1981view} and \cite{nocedal2006numerical}.
Note that \eqref{wolfe2ls} is replaced by
\begin{equation*}
\abs{\phi_k'(\alpha_k)}\le c_2\abs{\phi_k'(0)}
\end{equation*}
in the strong Wolfe conditions.

In the following numerical experiments, unless otherwise noted, we set $c_1=10^{-4}$ and $c_2=0.1$ and use line search algorithms for the weak and strong Wolfe conditions based on \cite{lemarechal1981view} and \cite{nocedal2006numerical}, respectively.
See Figure 1 in \cite{lemarechal1981view} and Algorithms 3.5 and 3.6 in \cite{nocedal2006numerical} for further details.

In the algorithm in \cite{lemarechal1981view}, we must interpolate a lower bound $t_{{\rm L}}$ and an upper bound $t_{{\rm R}}$ of an appropriate step size to find $t_{{\rm interpol}}$.
Furthermore, when no upper bound is available, we have to extrapolate the lower bound $t_{{\rm L}}$ to find $t_{{\rm extrapol}}$.
To this end, we use the simple formula
\begin{equation*}
t_{{\rm interpol}}:=\frac{t_{{\rm L}}+t_{{\rm R}}}{2},\quad t_{{\rm extrapol}}:=2t_{{\rm L}}.
\end{equation*}

In Algorithm 3.5 in \cite{nocedal2006numerical}, we have to extrapolate an estimated step size $\alpha^{(i)}$ to determine a value $\alpha^{(i+1)}$, where the superscript is the inner iteration number of Algorithm 3.5 in \cite{nocedal2006numerical}.
We use the formula introduced in \cite{nocedal2006numerical}:
\begin{equation}
\alpha:=\alpha^{(i)}-(\alpha^{(i)}-\alpha^{(i-1)})\frac{\phi'(\alpha^{(i)})+d_2-d_1}{\phi'(\alpha^{(i)})-\phi'(\alpha^{(i-1)})+2d_2},\label{extra}
\end{equation}
with
\begin{align*}
d_1=&\phi'(\alpha^{(i-1)})+\phi'(\alpha^{(i)})-3\frac{\phi(\alpha^{(i-1)})-\phi(\alpha^{(i)})}{\alpha^{(i-1)}-\alpha^{(i)}},\\
d_2=&\sign(\alpha^{(i)}-\alpha^{(i-1)})\sqrt{d_1^2-\phi'(\alpha^{(i-1)})\phi'(\alpha^{(i)})},
\end{align*}
where $\sign$ is the sign function.
This formula is based on the cubic interpolation of $\alpha^{(i-1)}$ and $\alpha^{(i)}$.
Note that the $\alpha$ obtained in \eqref{extra} may be too small or too large.
To safeguard this, we use a rule introduced in \cite{fletcher2013practical}, that is, we choose $\alpha^{(i+1)}$ from the interval $[2\alpha^{(i)}-\alpha^{(i-1)}, \alpha^{(i)}+9(\alpha^{(i)}-\alpha^{(i-1)})]$.
Together with Eq.~\eqref{extra}, this rule gives the formula
\begin{equation}
\alpha^{(i+1)}:=\min\{\max\{\alpha, 2\alpha^{(i)}-\alpha^{(i-1)}\}, \alpha^{(i)}+9(\alpha^{(i)}-\alpha^{(i-1)})\}.\label{extra2}
\end{equation}
We implement Algorithms 3.5 and 3.6 in \cite{nocedal2006numerical} using \eqref{extra2}.

\subsection{Rayleigh quotient minimization on the sphere}
Let $A$ be an $n\times n$ symmetric matrix.
Throughout this section, we consider the problem of minimizing 
\begin{equation*}
f(x)=x^TAx
\end{equation*}
on the unit sphere $S^{n-1}=\{x\in\mathbb R^n\,|\,x^Tx=1\}$ endowed with the natural Riemannian metric
\begin{equation*}
\langle \xi,\eta\rangle_x = \xi^T\eta,\quad \xi, \eta \in T_x S^{n-1},\ x\in S^{n-1},
\end{equation*}
in which case the gradient of the objective function $f$ is written as
\begin{equation*}
\grad f(x)=2(I_n-xx^T)Ax.
\end{equation*}
For simplicity, we use a retraction $R$ defined by
\begin{equation*}
R_x(\xi)=\frac{x+\xi}{\norm{x+\xi}},\quad \xi\in T_x S^{n-1},\ x\in S^{n-1},
\end{equation*}
where $\norm{\cdot}$ denotes the Euclidean norm, and use the associated differentiated retraction \cite{AbsMahSep2008}
\begin{align*}
&\mathcal{T}_{\eta}(\xi)=\D R_x(\eta)[\xi]\notag\\
=&\frac{1}{\norm{x+\eta}}\left(I_n-\frac{1}{\norm{x+\eta}^2}(x+\eta)(x+\eta)^T\right)\xi,\quad \eta,\xi\in T_x S^{n-1},\ x\in S^{n-1}.
\end{align*}

For means of reproducibility, we first set $A=\diag(1,2,\ldots,n)$, $x_0={\bm 1}_n/\sqrt{n}$, and $\varepsilon = 10^{-5}$ without using random values, where ${\bm 1}_n$ is an $n$-dimensional vector whose elements are all equal to $1$.
We then apply conjugate gradient methods with $\beta^{\FR}$ and $\beta^{\DY}$
under the weak and strong Wolfe conditions,
and with $n=100$ and $n=500$.
\begin{figure}[htbp]
  \begin{center}
   \includegraphics[width=100mm]{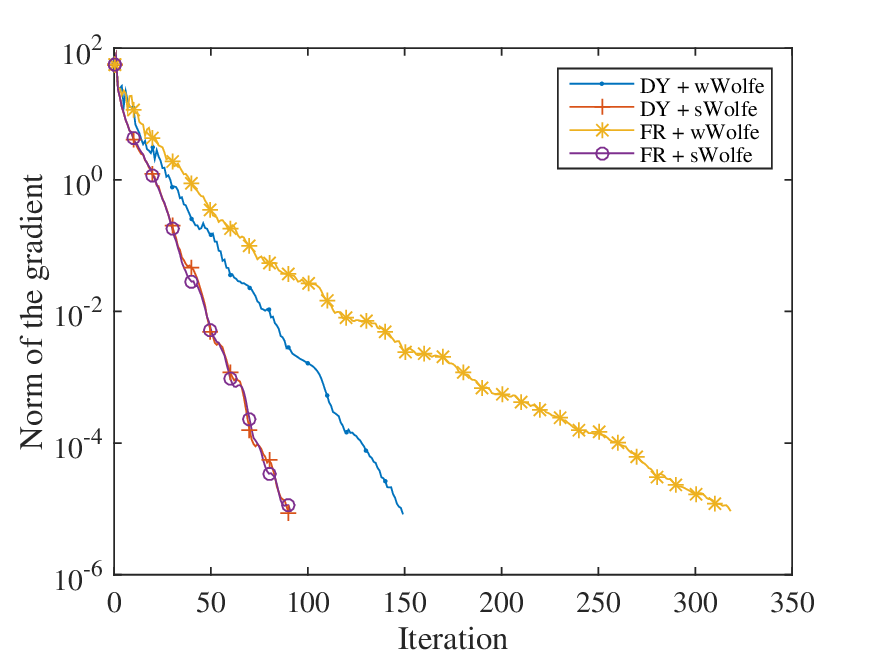}
  \end{center}
  \caption{The sequence of $\norm{\grad f(x_k)}_{x_k}$ evaluated on the sequence $\left\{x_k\right\}$ generated by Algorithm \ref{CGMDY} with $n=100$.}
  \label{fig1}
\end{figure}
\begin{figure}[htbp]
  \begin{center}
   \includegraphics[width=100mm]{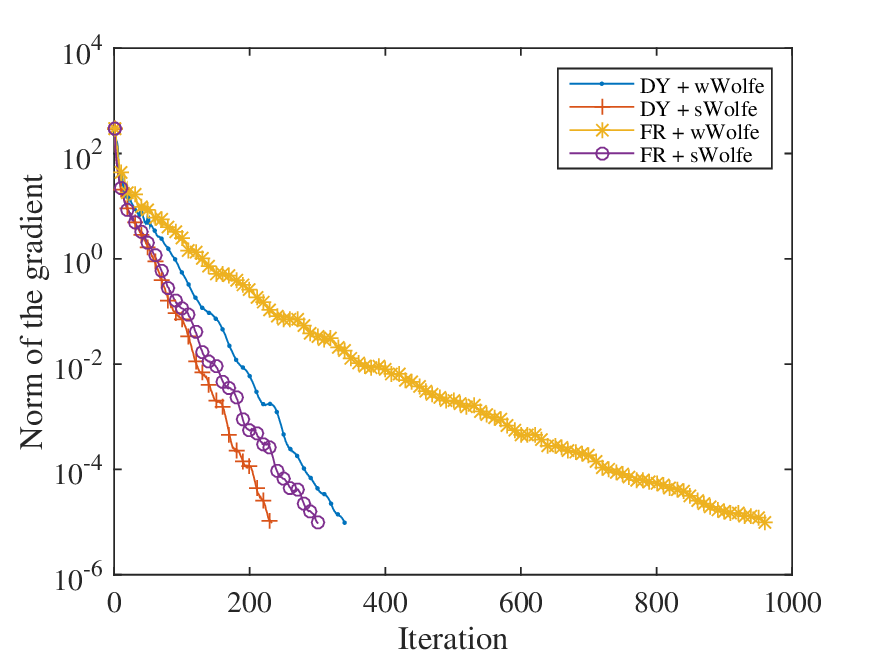}
  \end{center}
  \caption{The sequence of $\norm{\grad f(x_k)}_{x_k}$ evaluated on the sequence $\left\{x_k\right\}$ generated by Algorithm \ref{CGMDY} with $n=500$.}
  \label{fig2}
\end{figure}
Figure~\ref{fig1} (for $n=100$) and Fig.~\ref{fig2} (for $n=500$) show the results of the experiments.
In Fig.~\ref{fig2}, the graph corresponding to the method with $\beta^{\FR}$ and the weak Wolfe conditions implies that this method is much slower than the others.
In fact, with some other initial points, no appropriate step size satisfying the weak Wolfe conditions can be found for this method.
For example, with $x_0=({\bm 1}_{35}^T \ \  {\bm 0}_{465}^T)^T/\sqrt{35}$, we have $\langle \grad f(x_k), \eta_k\rangle_{x_k}=1.2646\times 10^{-4}>0$ at $k=37$,
which means that the search direction $\eta_k$ is not a descent direction for $f$ at $x_k$.
In contrast, under Prop.~\ref{lem1}, the method with $\beta^{\DY}$ and the weak Wolfe conditions is guaranteed to have the property $\langle \grad f(x_k), \eta_k\rangle_{x_k}<0$ for all $k\ge 0$.
The number of iterations, function evaluations, and gradient evaluations, as well as the average computational time (1000 times) required for convergence, are shown in Table \ref{table1} (for $n=100$) and Table \ref{table2} (for $n=500$).
Note that we have abbreviated the weak and strong Wolfe conditions as wWolfe and sWolfe, respectively.

\begin{table}[htbp]
\begin{center}
\caption{Comparison of several quantities between different conjugate gradient methods with $n=100$.}
\scriptsize 
\begin{tabular}{ccccc}
\hline
\backslashbox{Method}{} &  Iterations & Function Evals. & Gradient Evals. & Computational time  \\ \hline
\vspace{-2.5mm}
\\
DY + wWolfe & $149$ & $210$ & $206$ & $0.0175$\\ 
DY + sWolfe & $90$ & $288$ & $244$ & $0.0187$\\
FR + wWolfe & $318$ & $619$ & $577$ & $0.0429$\\
FR + sWolfe & $91$ & $293$ & $258$ & $0.0191$\\ \hline
\end{tabular}
\label{table1}
\end{center}
\end{table}

\begin{table}[htbp]
\begin{center}
\caption{Comparison of several quantities between different conjugate gradient methods with $n=500$.}
\scriptsize 
\begin{tabular}{ccccc}
\hline
\backslashbox{Method}{} &  Iterations & Function Evals. & Gradient Evals. & Computational time  \\ \hline
\vspace{-2.5mm}
\\
DY + wWolfe & $340$ & $373$ & $367$ & $0.0522$\\ 
DY + sWolfe & $232$ & $657$ & $467$ & $0.0658$\\
FR + wWolfe & $960$ & $1902$ & $1757$ & $0.1988$\\
FR + sWolfe & $300$ & $723$ & $529$ & $0.0730$\\ \hline
\end{tabular}
\label{table2}
\end{center}
\end{table}

These results indicate that we should not use the Fletcher--Reeves-type method with the weak Wolfe conditions, because it is not always guaranteed to generate a converging sequence.
In Tables \ref{table1} and \ref{table2}, we can observe that ``DY + wWolfe'' requires the shortest computational time.
In the following, we compare three methods, namely ``DY + wWolfe'', ``DY + sWolfe'', and ``FR + sWolfe'', using more general $1000$ problems.

We next set $n=100$, generate $1000$ symmetric matrices $A$ and initial points $x_0$ with randomly chosen elements, and then apply the above three methods.
Table \ref{table3} shows the average values of the quantities in Tables \ref{table1} and \ref{table2} over the $1000$ matrices $A$.

\begin{table}[htbp]
\begin{center}
\caption{Comparison of the averages of several quantities between different conjugate gradient methods with $1000$ randomly chosen matrices $A$.}
\scriptsize 
\begin{tabular}{ccccc}
\hline
\backslashbox{Method}{} &  Iterations & Function Evals. & Gradient Evals. & Computational time  \\ \hline
\vspace{-2.5mm}
\\
DY + wWolfe & $242.751$ & $538.177$ & $469.628$& $0.0334$\\ 
DY + sWolfe & $160.270$ & $529.736$& $410.278$& $0.0322$\\
FR + sWolfe & $201.441$& $649.900$& $513.879$& $0.0390$\\ \hline
\end{tabular}
\label{table3}
\end{center}
\end{table}

Even though ``DY + wWolfe'' requires the most iterations, it generates fewer function and gradient evaluations per iteration.
In our experiments, ``DY + sWolfe'' had the shortest average computational time.
In fact, however, ``DY + sWolfe'' was not the fastest method for all $1000$ matrices, as ``DY + wWolfe'' converged faster with $379$ of the $1000$ matrices.
Furthermore, we can observe that the Dai--Yuan-type Riemannian conjugate gradient method is more efficient than the Fletcher--Reeves-type method.

We also perform the same experiments as in Table \ref{table3} with other choices of $c_1$ and $c_2$ in the (strong) Wolfe conditions:
$(c_1, c_2) = (10^{-4}, 10^{-2}), (10^{-4}, 0.2), (10^{-4}, 0.3), (10^{-4}, 0.4), (0.1, 0.2),$ $(0.1, 0.3), (0.1, 0.4)$.
The results show that the choice of $c_1$ does not affect the performance of the proposed method very much,
while the choice of $c_2$ can improve or worsen the performance.
In fact, in our experiments, the proposed method with $(c_1, c_2)=(0.1, 0.4)$ was the fastest.
The three methods ``DY + wWolfe'', ``DY + sWolfe'', and ``FR + sWolfe'' with $(c_1, c_2)=(0.1, 0.4)$ required $0.0255$, $0.0310$, and $0.0465$ seconds (on average) for convergence, respectively.
In this case, ``DY + wWolfe'' converged faster than ``DY + sWolfe'' with $842$ of the $1000$ matrices.
Even though a better choice of $c_1$ and $c_2$ can improve conjugate gradient methods,
we observe that the Dai--Yuan-type method is always more efficient than the Fletcher--Reeves-type method.

\section{Concluding Remarks}\label{sec6}
We have proposed a new Riemannian conjugate gradient method based on the Dai--Yuan-type Euclidean conjugate gradient algorithm.
To generalize $\beta_k^{\DY}=\nabla f(x_k)^T\nabla f(x_k)/\eta_{k-1}^Ty_k$ in the Dai--Yuan algorithm, we used another expression $\beta^{\DY}_k=\nabla f(x_k)^T\eta_k/\nabla f(x_{k-1})^T\eta_{k-1}$ and the notion of a vector transport.
We thus proposed a new $\beta_k$, and hence a new Riemannian conjugate gradient method.

We have proved that the proposed algorithm is well defined and that sequences generated by the algorithm are globally convergent.
The notion of a scaled vector transport plays an essential role in the convergence analysis,
that is, property \eqref{Tkineq} leads to a contradiction with Zoutendijk's theorem
if we assume that the assertion of Thm.~\ref{globalthm} is false.
One advantage of our algorithm is that we do not have to search for step sizes satisfying the strong Wolfe conditions, which are necessary conditions for the existing Fletcher--Reeves-type algorithm, and need only compute the step sizes with the weak Wolfe conditions.

Through a series of numerical experiments, we demonstrated that the Dai--Yuan-type method is better than the Fletcher--Reeves-type approach.
Implementing the Dai--Yuan-type method with the weak Wolfe conditions is not always better than that with the strong Wolfe conditions, but is preferable for many problems.
Nevertheless, we can choose which conditions we use according to the form of the objective function, the size of the problem, and so on.

It is also worth pointing out that the role of the scaled vector transport in Algorithm \ref{CGMDY} may be imposed on $\beta_k$.
Putting the scaling factor $\min\left\{1, \norm{\eta_k}_{x_k}/\norm{\mathcal{T}^R_{\alpha_k\eta_k}(\eta_k)}_{x_{k+1}}\right\}$ on $\beta_{k+1}$, we can rewrite the computations in Step 5 of Algorithm \ref{CGMDY} as
\begin{equation*}
c^{(k)}=\min\left\{1, \frac{\norm{\eta_k}_{x_k}}{\norm{\mathcal{T}^R_{\alpha_k\eta_k}(\eta_k)}_{x_{k+1}}}\right\},
\end{equation*}
\begin{equation*}
\bar\beta_{k+1}=\frac{c^{(k)}\norm{\grad f(x_{k+1})}_{x_{k+1}}^2}{c^{(k)}\langle\grad f(x_{k+1}), \mathcal{T}^{R}_{\alpha_{k}\eta_{k}}(\eta_{k})\rangle_{x_{k+1}}-\langle\grad f(x_{k}),\eta_{k}\rangle_{x_{k}}},
\end{equation*}
\begin{equation*}
\eta_{k+1}=-\grad f(x_{k+1})+\bar\beta_{k+1}\mathcal{T}^{R}_{\alpha_k\eta_k}(\eta_k).
\end{equation*}
Although these formulas stem from the notion of a scaled vector transport, 
they can be regarded as if we are not using a scaled vector transport, but are simply applying another type of $\beta_k$ with a usual vector transport.
These may be easier to deal with in the common framework of the conjugate gradient method discussed in \cite{AbsMahSep2008}.

\section*{Acknowledgments}
The author would like to thank the anonymous referees for their valuable comments that helped improve the paper significantly.
This work was supported by JSPS (Japan Society for the Promotion of Science) KAKENHI Grant Number 26887037.

\bibliography{refCOA.bib}

\begin{thebibliography}{10}

\bibitem{AbsMahSep2008}
{\sc P.-A. Absil, R.~Mahony, and R.~Sepulchre}, {\em Optimization Algorithms on
  Matrix Manifolds}, Princeton University Press, Princeton, 2008.

\bibitem{dai1999nonlinear}
{\sc Y.-H. Dai and Y.~Yuan}, {\em A nonlinear conjugate gradient method with a
  strong global convergence property}, SIAM J. Optim., 10 (1999), pp.~177--182.

\bibitem{edelman1998geometry}
{\sc A.~Edelman, T.~A. Arias, and S.~T. Smith}, {\em The geometry of algorithms
  with orthogonality constraints}, SIAM J. Matrix Anal. Appl., 20 (1998),
  pp.~303--353.

\bibitem{fletcher2013practical}
{\sc R.~Fletcher}, {\em Practical Methods of Optimization}, Wiley, New York,
  2013.

\bibitem{fletcher1964function}
{\sc R.~Fletcher and C.~M. Reeves}, {\em Function minimization by conjugate
  gradients}, Comput. J., 7 (1964), pp.~149--154.

\bibitem{hestenes1952methods}
{\sc M.~R. Hestenes and E.~Stiefel}, {\em Methods of conjugate gradients for
  solving linear systems}, J. Res. Natl. Bur. Stand., 49 (1952), pp.~409--436.

\bibitem{lemarechal1981view}
{\sc C.~Lemar{\'e}chal}, {\em A view of line-searches}, in Optimization and
  Optimal Control, Springer, 1981, pp.~59--78.

\bibitem{nocedal2006numerical}
{\sc J.~Nocedal and S.~Wright}, {\em Numerical Optimization, Series in
  Operations Research and Financial Engineering}, Springer, New York, 2006.

\bibitem{ring2012optimization}
{\sc W.~Ring and B.~Wirth}, {\em Optimization methods on {R}iemannian manifolds
  and their application to shape space}, SIAM J. Optim., 22 (2012),
  pp.~596--627.

\bibitem{sato2013cg}
{\sc H.~Sato and T.~Iwai}, {\em A new, globally convergent {R}iemannian
  conjugate gradient method}, Optimization, 64 (2015), pp.~1011--1031.

\end{thebibliography}
\end{document}